\documentclass{amsart}

\usepackage{amssymb}
\usepackage{eqlist, color, mathtools}
\usepackage{mathrsfs, textcomp}
\usepackage{graphicx}
\usepackage{bm}
\usepackage{enumitem}

\usepackage{amsmath,amsfonts,amsthm,latexsym}
\usepackage[all]{xy}

\newtheorem{theorem}{Theorem}[section]
\newtheorem{lemma}[theorem]{Lemma}
\newtheorem{proposition}[theorem]{Proposition}

\theoremstyle{definition}
\newtheorem{definition}[theorem]{Definition}
\newtheorem{example}[theorem]{Example}

\newtheorem{remark}[theorem]{Remark}
\newtheorem{notation}[theorem]{Notation}

\newtheorem{observations}[theorem]{Observations}

\numberwithin{equation}{section}

\newcommand{\C}{\mathbb{C}}

\newcommand{\N}{\mathbb{N}}

\newcommand{\R}{\mathbb{R}}
\newcommand{\Z}{\mathbb{Z}}

\newcommand{\cI}{\mathcal{I}}

\newcommand{\cM}{\mathcal{M}}

\newcommand{\cP}{\mathcal{P}}
\newcommand{\cR}{\mathcal{R}}

\newcommand{\fC}{{\mathfrak C}}

\newcommand{\sC}{\mathscr{C}}

\newcommand{\sZ}{\mathscr{Z}}

\newcommand{\ssm}{\smallsetminus}
\newcommand{\pperp}{\perp\perp}

\DeclareMathOperator{\Min}{Min}

\DeclareMathOperator{\Max}{Max}

\DeclareMathOperator{\coz}{coz}

\DeclareMathOperator{\Cl}{cl}

\DeclareMathOperator{\Spec}{Spec}
\DeclareMathOperator{\id}{id}

\DeclareMathOperator{\Val}{Val}
\DeclareMathOperator{\Clop}{Clop}

\DeclareMathOperator{\M}{(M)}
\DeclareMathOperator{\A}{(Ab)}
\DeclareMathOperator{\Y}{(Y)}
\DeclareMathOperator{\Arch}{(Arch)}
\DeclareMathOperator{\Idl}{Id}
\DeclareMathOperator{\idl}{id}

\newcommand{\ra}{\rightarrow}
\newcommand{\Mart}{Mart\'{i}nez }
\newcommand{\Marti}{Mart\'{i}nez}

\theoremstyle{definition}
\theoremstyle{definition}
\theoremstyle{definition}
\theoremstyle{definition}
\theoremstyle{definition}
\theoremstyle{definition}
\theoremstyle{remark}
\theoremstyle{definition}
\theoremstyle{definition}

\begin{document}
\title[$\M$, $\Y$, and $G+B$]{Semi-boolean and Yosida $\ell$-groups, Mart\'{i}nez and Yosida frames, and the $G+B$ construction}












\author{Papiya Bhattacharjee, \quad Anthony W. Hager, \quad Warren Wm. McGovern, \quad Brian Wynne   }

\address{Department of Mathematical Sciences, Charles E. Schmidt College of Science, Florida Atlantic University, Boca Raton, FL 33431, USA}
\email{pbhattacharjee@fau.edu (P. Bhattacharjee)}

\address{Department of Mathematics and Computer Science, Wesleyan University, Middletown, CT 06459, USA}
\email{ahager@wesleyan.edu (A. W. Hager)}

\address{Wilkes Honors College, Florida Atlantic University, Jupiter, FL 33458, USA}
\email{warren.mcgovern@fau.edu (W. Wm. McGovern)}

\address{Department of Mathematics, Lehman College, CUNY, Gillet Hall, Room 211, 250 Bedford Park Blvd. West, Bronx, NY 10468, USA}
\email{brian.wynne@lehman.cuny.edu (B. Wynne)}

\subjclass[2020]{06F20, 46E25, 06D22, 08C05, 54D80}
\keywords {lattice-ordered group, prime subgroup, $d$-subgroup, principal polar}



\begin{abstract}
The class of semi-boolean $\ell$-groups was introduced in 1968 by A. Bigard. These are the $\ell$-groups $G$ in which the principal convex $\ell$-subgroup $G(a)$ generated by any $a \in G$ is equal to the polar $a^{\pperp}$. Examples include all hyperarchimedean $\ell$-groups and all existentially closed abelian $\ell$-groups. Ordered by inclusion, the set of convex $\ell$-subgroups of a semi-boolean $\ell$-group is a \Mart frame (an algebraic frame with FIP in which every element is a $d$-element). Related are the Yosida $\ell$-groups, i.e., the $\ell$-groups whose frame of convex $\ell$-subgroups is a Yosida frame (an algebraic frame with FIP in which every compact element is a meet of maximal elements). Applying results on \Mart frames and Yosida frames, we obtain new characterizations of the semi-boolean and Yosida $\ell$-groups, show that the former constitute a radical class and the latter do not,  and present new examples with special properties. To build some of our examples, we introduce the $G+B$ construction for $\ell$-groups, an adaptation of the $A+B$ construction from commutative algebra.
\end{abstract}


\maketitle

\thispagestyle{empty}

\section{Introduction}\label{sect0}

Let $G$ be a lattice-ordered group ($\ell$-group, for the sake of abbreviation). A subgroup $H$ of $G$ is called a $z$-subgroup if $a^{\pperp} \subseteq H$ for every $a \in H$. In his 1968 note \cite{bigard}, Bigard introduces the notion of a $z$-subgroup and announces the equivalence of the following three statements:
\begin{itemize}
\item[(1)] For every $a \in G$, $G(a) = a^{\pperp}$.
\item[(2)] Every convex $\ell$-subgroup of $G$ is a $z$-subgroup.
\item[(3)] For all $a,b \in G$, $G(a) = G(b)$ if and only if $a^{\pperp} = b^{\pperp}$.
\end{itemize}
Moreover, he calls an $\ell$-group semi-boolean if it satisfies (1)--(3), and he observes that an $\ell$-group is hyperarchimedean if and only if it is semi-boolean and projectable. The equivalence of (1)--(3) is proved in  \cite[Proposition 3.3.10]{BKW77} and in \cite[Proposition 15.5]{darnel}, but neither of those books contains any further information about semi-boolean $\ell$-groups (and neither mentions them by name). Indeed, we are aware of no publication since \cite{bigard} that specifically analyzes the semi-boolean $\ell$-groups. In this paper, we hope to draw more attention to these $\ell$-groups by extending Bigard's characterization using results from frame theory, cataloging their basic properties, and exhibiting new examples. For reasons explained shortly, we use the symbol $\M$ to denote the class of semi-boolean $\ell$-groups.

One reason for our interest in the semi-boolean $\ell$-groups is that they play a significant role in the model theory of abelian $\ell$-groups. Condition (1) above is among the five axioms that make up Weispfenning's characterization of the existentially closed (e.c.) abelian $\ell$-groups (\cite[Theorem 3.6.3(2)]{W}), so every e.c. abelian $\ell$-group is semi-boolean. Moreover, that axiom is the only one that is not expressible by a first-order sentence in the language of $\ell$-groups, so in that sense it is the most complicated one. The information presented in this paper might eventually help to shed light on some of the long-standing open problems about e.c. abelian $\ell$-groups (see, e.g., \cite[Section 7]{scowcroft}).

Unlike semi-boolean $\ell$-groups, $z$-subgroups have been investigated by several authors, including  in \cite{B72} and \cite{cm} in the context of $\ell$-groups, in \cite{keimel} in the context of ideal lattices, and in \cite{HP80a}, \cite{HP80b}, \cite{dePagter}, and \cite{HP83} in the context of vector lattices (where they are called ``$d$-ideals"). Later, \Mart and Zenk \cite{mz2003}  studied $d$-elements  (analogues of $z$-subgroups) in algebraic frames, and more recently \cite{bkm} examined those algebraic frames with FIP in which every element is a $d$-element (an analogue of condition (2) above), which the present authors call \Mart frames.

Let $G$ be an $\ell$-group. Then $G$ is semi-boolean if and only if the frame $\fC(G)$ of all convex $\ell$-subgroups of $G$ is a \Mart frame. We use this connection to transfer results about \Mart frames to semi-boolean $\ell$-groups. To signal this  frame-theoretic approach we write $G \in \M$ to abbreviate that $\fC(G)$ is a \Mart frame (and hence that $G$ is semi-boolean). When they have enough maximal elements, \Mart frames are Yosida frames (in the sense of \cite{mz06}); if $\fC(G)$ is a Yosida frame, then we write $G \in \Y$. After Section \ref{prelim} gives some background on lattices and $\ell$-groups, Section \ref{MandY} recasts and extends results on \Mart and Yosida frames from \cite{bkm} and \cite{mz06}, which are then used in Sections \ref{sectionM} and \ref{sectionY} to produce new characterizations of $\M$ and $\Y$; it is also shown that $\M$ is a radical class and that an archimedean $G \in \M$ can have $G \notin \Y$. The final sections introduce and apply what we call the $G+B$ construction. Adapted from ring theory, this construction yields a retract of $G$ whose root system of  prime subgroups is closely related to that of $G$ and yet is sufficiently different to allow the extension obtained to have novel features.  Section \ref{G+B} provides a quick account of the basic properties of $G+B$, then Section \ref{MfromY} characterizes when $G+B \in \M$ and applies that result to produce an example of an archimedean $\ell$-group with strong unit that is in $\M$ but whose Yosida space not zero-dimensional.

\section{Preliminaries}\label{prelim}

We assume some familiarity with lattices and lattice-ordered groups; for undefined terms see \cite{gratzer} and \cite{darnel}.

\subsection{Lattices}

Let $L$ be a distributive 0-lattice. If $H \subseteq L$, then we write $\idl(H)$ for the ideal generated by $H$ in $L$ (\cite[p. 32]{gratzer}); if $H = \{ a \}$, then we write $\idl(a)$ instead of $\idl(H)$. We write $\Spec(L)$ for the set of prime ideals of $L$ (\cite[p. 116]{gratzer}), $\Min(L)$ for the set of minimal prime ideals of $L$ (\cite[p.196]{gratzer}), $\Max(L)$ for the set of maximal ideals of $L$ (it is possible that $\Max(L) = \emptyset$), and $\Idl L$ for the full ideal lattice of $L$ (\cite[p. 33]{gratzer}). If $S \subseteq L$, let $S^{\perp} = \{ b \in L : a \wedge b = 0 \mbox{ for all } a\in S \}$; if $S= \{ a \}$, we write $a^{\perp}$ instead of $S^{\perp}$. If $S \in \Idl L$, then $S^{\perp} \in \Idl L$. If $a^{\perp} = \{ 0 \}$, then we say that $a$ is dense in $L$. 

A frame is a complete lattice satisfying the infinite distributive law
\[
x \wedge \bigvee Y = \bigvee  \{ x \wedge y : y \in Y \},
\]
and an algebraic frame is one in which every element is a join of compact elements (\cite[Definition 41]{gratzer}). For example, $\Idl L$ is an algebraic frame (\cite[Theorem 42]{gratzer}). Moreover, the set of compact elements of $\Idl L$ is a sublattice of $\Idl L$ that is isomorphic to $L$ (\cite[Corollary 7]{gratzer}). We say a frame has FIP whenever the meet of any two compact elements is again compact. It follows that $F$ is an algebraic frame with FIP if and only if $F$ is isomorphic to $\Idl L$ for some distributive 0-lattice $L$.

We consider two different topologies on $\Spec(L)$. First, a base for the hull-kernel topology on $\Spec(L)$ is $\{ U(a) : a \in L \}$, where $U(a) = \{ p \in \Spec(L) : a \notin p \}$ (see, e.g., \cite[Ch. II.5]{gratzer}, where a more general version is developed for join-semilattices). Second, a base for the patch topology on $\Spec(L)$ is $\{ U(a) \cap V(b) : a,b \in L \}$, where $V(b) = \{ p \in \Spec(L) : b \in p \}$.


If $I \in \Idl L$ and $c^{\pperp} \subseteq I$ whenever $c \in L$, then we call $I$ a $d$-element of $\Idl L$. Suppose $F$ is an algebraic frame $F$ with FIP. Then $F$ is said to be a \Mart frame if every element of $F$ is a $d$-element, and $F$ is called a Yosida frame if every compact element of $F$ is a meet of maximal elements. See \cite{mz2003} and \cite{bkm} for more about $d$-elements and \Mart frames, and see \cite{mz06} for more about Yosida frames. 

\subsection{General $\ell$-groups}



Let $G$ be a lattice-ordered group ($\ell$-group, for the sake of abbreviation). If $a \in G$, then the principal convex $\ell$-subgroup of $G$ generated by $a$ is denoted by $G(a)$, and $G(a) = \{ b \in G : \vert b \vert \leq n \vert a \vert \mbox{ for some } n \in \N \}$  (\cite[Proposition 7.13]{darnel}). Let $\mathbf{C}(G)$ be the set of all principal convex $\ell$-subgroups of $G$ ordered by inclusion. Then $\mathbf{C}(G)$ is a sublattice of the lattice $\fC(G)$ of all convex $\ell$-subgroups of $G$  (\cite[Proposition 7.15]{darnel}). In fact, $\mathbf{C}(G)$ is precisely the set of compact elements of $\fC(G)$  (\cite[Proposition 7.16]{darnel}), and $\fC(G)$ is isomorphic to $\Idl \mathbf{C}(G)$, and thus is an algebraic frame with FIP. Moreover, the spaces $\Spec(G)$, $\Min(G)$, and $\Max(G)$ of proper prime subgroups, minimal prime subgroups, and maximal convex $\ell$-subgroups  of $G$ correspond to $\Spec(L)$, $\Min(L)$, and $\Max(L)$, respectively, where $L \coloneqq \mathbf{C}(G)$ (e.g., if $p \in \Spec(G)$, then $\{ G(a) : a \in p \} \in \Spec(L) \}$, and if $q \in \Spec(L)$ then $\bigcup q \in \Spec(G)$).

Recall that $p \in \Spec(G)$ if and only if $G/p$ is totally ordered (\cite[Theorem 9.1]{darnel}). One can show that $\Spec(G) \neq \emptyset$ using Zorn's lemma. In particular, for any $0\neq a \in G$, there are convex $\ell$-subgroups which are maximal with respect to not containing $a$. These are called values of $a$; we denote the set of all values of $a$ by $\Val(a)$. Note that $\Val(a) \subseteq \Spec(G)$ (\cite[Corollary 10.4]{darnel}). In general, not every prime subgroup is a value of some element of $G$. With respect to the hull-kernel topology, $\Val(a)$ is a compact Hausdorff space.



By an application of Zorn's Lemma, every prime subgroup contains at least one minimal prime subgroup. On the other hand, $\Max(G) = \emptyset$ is possible even when $G$ is abelian. In the abelian case, $m \in \Max(G)$ if and only if $G/m$ is $\ell$-isomorphic to a subgroup of the real numbers; H\"{o}lder's Thereom \cite[Theorem 24.16]{darnel} says that every totally ordered archimedean $\ell$-group is isomorphic to a subgroup of $\mathbb{R}$. Both $\Min(G)$ and $\Max(G)$ can be viewed as subspaces of $\Spec(G)$ with  respect to either the hull-kernel or the patch topology.

Recall that for a subset $S\subseteq G$, the \emph{polar} of $S$ is the set
$$S^\perp=\{h\in G: |h|\wedge |s|=0 \hbox{ for all } s\in S\},$$
which is a convex $\ell$-subgroup of $G$. Under inclusion, the set of all polars forms a complete Boolean algebra with $\perp$ as the complement operation (\cite[Theorem 13.7]{darnel}).
When $S=\{a\}$ we write instead $a^\perp$. This notation fits with that discussed above for lattices: if $a \in G$, then we have $a^{\perp} = \{ b \in G : G(b) \in G(a)^{\perp} \}$, where $G(a)^{\perp}$ is the $\perp$ operation in $\mathbf{C}(G)$. The collection of all principal polars, i.e., polars of the form $a^{\pperp}$, forms a sublattice of the Boolean algebra of all polars (\cite[Theorem 13.11]{darnel}). When $u^\perp=\{0\}$ and $u \geq 0$, we call $u$ a weak (order) unit. Equivalently, $u\in G^+$ is a weak unit if $G=u^{\pperp}$. Note that $u$ is a weak unit if and only if $G(u)$ is dense in $\mathbf{C}(G)$. A strong (order) unit is an element $0 \leq u \in G$, such that $G(u)=G$. Neither type of unit need exist in $G$. A principal polar is a convex $\ell$-subgroup of the form $a^{\pperp}$ for some $a\in G$. A $d$-subgroup of $G$ is a $d$-element of $\fC(G)$, and the collection of all prime subgroups that are $d$-subgroups is denoted $\Spec_d(G)$.

\subsection{The categories $\mathbf{W}$ and $\mathbf{W}^*$}

Recall that the $\ell$-group $G$ is said to be archimedean if $a,b \geq 0$ and $na\leq b$ for all $n\in \N$ imply that $a = 0$; in that case we write $G \in \Arch$. The objects of the category {\bf W} are the pairs $(G,u)$ with $G \in \Arch$ and $0 \leq u \in G$ a distinguished weak unit, and the morphisms are arrows $f:(G,u)\ra (H,v)$ such that $f:G\ra H$ is an $\ell$-homomorphism, i.e., a group homomorphism and a lattice homomorphism, with $f(u)=v$. $\bf{W}^*$ denotes the full subcategory of $\bf{W}$ consisting of those objects whose distinguished unit is strong. In both {\bf W} and $\bf{W}^*$, we let $YG=Y(G,u) = \Val(u)$ and call this the {\it Yosida space} of $(G,u)$.

If $X$ is a Tychonoff space, then $C(X)$ denotes the collection of real-valued continuous functions on $X$. With pointwise operations and order, $C(X)$ is an archimedean $\ell$-group and $(C(X), \mathbf{1}) \in \bf{W}$, where $\mathbf{1}$ is the constant function with with value 1. We let $\overline{\R}=\R\cup \{\pm \infty\}$ be the two-point compactification of the real numbers $\mathbb{R}$. A continuous function $f \colon X \to \overline{\R}$ is said to be {\it almost real-valued} if $f^{-1}(\R)$ is dense in $X$. We let
$$D(X)=\{f:X\ra \bar{\R}: f \hbox{ is almost real-valued}\}.$$
In general, $D(X)$ is always a lattice but need not be a group. And clearly there are subsets of $D(X)$ which are $\ell$-groups, e.g., $C(X)$; we call such a subset an $\ell$-group in $D(X)$.

We now remind the reader of a useful representation for $\bf{W}$-objects (see \cite{HR77} for more on this representation).

\begin{theorem}[The Yosida Embedding Theorem]\label{YRT}
Let $(G,u) \in {\bf W}$. There is a {\bf W}-morphism $\varphi: (G,u) \ra (D(YG), {\bf 1})$ such that for any $p\in YG$ and closed $V\subseteq YG$ not containing $p$, there is some $a\in G$ such that $\varphi(a)(p)=0$ and $\varphi(a)(q)=1$ for all $q\in V$. Moreover, $YG$ is the unique compact Hausdorff space (up to homeomorphism) with this property.
\end{theorem}

It follows from the Yosida Embedding Theorem that an archimedean $\ell$-group with a strong order unit can be viewed as an $\ell$-subgroup of $C(YG)$ with $u$ as the constant function {\bf 1}.

When dealing with the Yosida embedding of {\bf W}-objects, we drop the use of $\varphi$ and identify $G$ with its image in $D(YG)$. Accordingly, for $a\in G$, we let
$$\coz(a)=\{p\in YG: a(p)\neq 0\} \hspace{.2in} \hbox{ and } \hspace{.2in} Z(a)=\{p\in YG: a(p)=0\}$$
and call these the {\it cozero-set} and {\it zero-set} of $a$, respectively. Note that the cozero-sets of $G$ form a base for the hull-kernel topology on $YG$ (as $\coz(a) = U(a) \cap YG$). Note too the following useful characterization of principal polars in $\bf{W}$-objects (where $\Cl \coz(a)$ denotes the closure of $\coz(a)$ in $YG$):

\begin{lemma}\label{polarYG}
If $(G,u)\in {\bf W}$ and $a \in G$, then
$a^{\pperp}=\{ b\in G: \coz(b)\subseteq \Cl \coz(a)\}.$
\end{lemma}

\begin{notation}\label{zsetnotation}
Suppose $\varphi \colon G \to \prod_{i \in I} G_i$ is an embedding of $G$ into a direct product of totally ordered groups. Since this representation of $G$ (identifying the elements of $G$ with their images in the product) need not be the Yosida representation, to avoid potential confusion we write 
\[
\sC_{\varphi}(a)=\{i\in I: \varphi(a)(i)\neq 0\} \mbox{ and } \sZ_{\varphi}(a) = I \ssm \sC_{\varphi}(a)
\]
instead of $\coz(a)$ and $Z(a)$, respectively. In this situation, for $a \in G$, let $I(\sZ_{\varphi}(a))$ be the set of all $b \in G$ such that $\sZ_{\varphi}(b) \supseteq \sZ_{\varphi}(a)$. Note that $I(\sZ_{\varphi}(a))$ is a convex $\ell$-subgroup of $G$ and that $G(a) \subseteq I(\sZ_{\varphi}(a)) \subseteq a^{\pperp}$ always. We sometimes omit the subscript $\varphi$ when the embedding is clear from context.
\end{notation}

\section{\Mart and Yosida frames}\label{MandY}

In this section we rephrase and extend some results about frames from earlier papers, putting them in a form that is convenient for application to $\ell$-groups.

Let $L$ be a distributive 0-lattice. Call $L$ disjunctive if for all $a,b \in L$, if $a \nleq b$, there is $c \in L$ such that $a \wedge c > 0$ and $b \wedge c = 0$. This property goes back to Wallman \cite{wallman}.  

Parts (2)-(8) of the following theorem are just \cite[Theorems 2.1, 2.5]{bkm} restated in terms of the frame of ideals of a distributive 0-lattice.

\begin{theorem}\label{martframes}
Suppose $L$ is a distributive $0$-lattice. Then following statements are equivalent
\begin{itemize}
\item[(1)] $L$ is disjunctive.
\item[(2)] $\Idl L$ is a Martinez frame.
\item[(3)] $\Spec_d(L) = \Spec(L)$.
\item[(4)] For all $c \in L$, $\idl(c) =\idl(c)^{\pperp}$.
\item[(5)] For all $c \in L$, $\idl(c)$ is a $d$-element of $\Idl L$.
\item[(6)] The set $\Min(L)$ is patch-dense in $\Spec(L)$.
\item[(7)] The set $\{ U(c) : 0 < c \in L \}$ is a $\pi$-base for the patch topology on $\Spec(L)$.
\item[(8)] Distinct compact open subsets of $\Spec(L)$ have distinct closures (with respect to the hull-kernel topology).
\item[(9)] For all $c \in L$, $\idl(c)$ is disjunctive.
\item[(10)] For all $I \in \Idl L$, $I$ is disjunctive.
\item[(11)] For all $c \in L$, $\idl(c) = \bigcap \{ p \in \Min(L) : c \in p \}$.  
\item[(12)] For $a,b \in L$, $\idl(a) = \idl(b)$ if and only if $\idl(a)^{\pperp} = \idl(b)^{\pperp}$.
\end{itemize}
\end{theorem}

\begin{proof}
(1) $\Rightarrow$ (2). Suppose $L$ is disjunctive. Since $\Idl L$ is an algebraic frame with FIP, to show that $\Idl L$ is Mart\'{i}nez, it suffices to show that every element in $\Idl L$ is a $d$-element. Take $I \in \Idl L$ and suppose $b \in I$. To show that $I$ is a $d$-element, it is enough to show $b^{\pperp} \subseteq \idl(b)$. Suppose $a \notin \idl(b)$. Since $L$ is disjunctive, there is $c \in L$ such that $a \wedge c > 0$ and $b \wedge c = 0$. It follows that $a \notin b^{\pperp}$. Hence $\Idl L$ is \Marti.

(2) $\Rightarrow$ (1).  Suppose $\Idl L$ is \Marti. Then by \cite[Theorem 2.1]{bkm} we have $\idl(c) = \idl(c)^{\pperp}$ for all $c \in L$. Suppose $a \nleq b$ in $L$. Then $a \notin \id(b) = \idl(b)^{\pperp}$, so there is $c \in L$ such that $a \wedge c > 0$ and $b \wedge c = 0$. Thus $L$ is disjunctive.

The equivalence of (1)-(8) now follows from \cite[Theorems 2.1,2.5]{bkm}.

(1) $\Rightarrow$ (9). Suppose $a,b \in \idl(c)$ with $a \nleq b$. If $L$ is disjunctive, there is $d \in L$ such that $a \wedge d > 0$ and $b \wedge d = 0$. Let $e = d \wedge c$. Then $e \in \idl(c)$, $a \wedge e = a \wedge d > 0$, and $b \wedge e  = b\wedge d = 0$. Thus $\idl(c)$ is disjunctive.

(9) $\Rightarrow$ (1). Suppose $a,b \in L$ with $a \nleq b$. If $\idl(a \vee b)$ is disjunctive, then there is $c \leq a \vee b$ such that $a \wedge c > 0$ and $b \wedge c = 0$. Thus $L$ is disjunctive.

(10) $\Rightarrow$ (9). Immediate.

(9) $\Rightarrow$ (10). Suppose $a,b \in I \in \Idl L$ with $a \nleq b$. If $\idl(a \vee b)$ is disjunctive, then there is $c \leq a \vee b \in I$ such that $a \wedge c > 0$ and $b \wedge c = 0$. Thus $I$ is disjunctive.

(11) $\Rightarrow$ (5). Every minimal prime element of $\Idl L$ is a $d$-element (this follows from \cite[Corollary 1.6]{bkm}), and any meet of $d$-elements is a $d$-element.

(4) $\Rightarrow$ (11). Every polar of $L$ is an intersection of minimal prime ideals of $L$ (this follows from \cite[Theorem 1.5]{kist}).

(4) $\Rightarrow$ (12). Immediate.

(12) $\Rightarrow$ (4). If $a \in \idl(c)^{\pperp} - \idl(c)$, then $\idl(a \vee c)^{\pperp} = \idl(c)^{\pperp}$ but $\idl(a \vee c) \neq \idl(c)$. So if (4) fails, then (12) fails.
\end{proof}

The following result, which is due to the anonymous reviewer, will be used in the next section to show that $\M$ is a radical class.

\begin{proposition}\label{reflemma}
Let $L$ be a distributive lattice with top 1 and bottom 0. If $u_0,u_1 \in L$, $u_0 \vee u_1 = 1$, and $\id(u_0)$ and $\id(u_1)$ are both disjunctive, then $L$ is disjunctive.
\end{proposition}

\begin{proof}
Suppose $a,b \in L$ with $a \nleq b$. Since $u_0 \vee u_1 = 1$, distributivity allows any $x \in L$ to be written in the form $x = (x \wedge u_0) \vee (x \wedge u_1)$. Thus, we may write $a = a_0 \vee a_1$ and $b = b_0 \vee b_1$ with $a_i \coloneqq a \wedge u_i$ and $b_i \coloneqq b \wedge u_i$ for $i \in \{ 0 ,1 \}$. Reindexing if necessary, we may assume $a_0 \nleq b_0$ (since $a_0 \leq b_0$ and $a_1 \leq b_1$ would contradict $a \nleq b$). By the assumption that $\id(u_0)$ is disjunctive, there is $c \leq u_0$ such that $0 < a_0 \wedge c$ and $0 = b_0 \wedge c$. Now, $b\wedge c = b\wedge (u_0 \wedge c) = b_0 \wedge c = 0$, and $a \wedge c = a \wedge (u_0 \wedge c) = a_0 \wedge c > 0$, showing that $L$ is disjunctive.
\end{proof}

Next we consider two conditions that are approximately dual to the disjunction property. A join-semilattice $L$ is said to be \emph{conjunctive} if it has $1$ and for all $a,b \in L$ such that $b \nleq a$, there is an element $c \in L$ such that $a \vee c < 1$ and $b \vee c = 1$. In the context of distributive 0-1-lattices, this is the order-dual of the disjunctive property; indeed, it has been called ``dual disjunctive" by some authors, e.g., \cite[Section 4]{cornish}. A join-semilattice $L$ is said to be \emph{ideally conjunctive} if for all $a,b \in L$ such that $b \nleq a$, there is $W \in \Idl L$ such that $a \in W \neq L = W \vee \id(b)$.
For background on these concepts see \cite{DIM}. 
In \cite[Section 3.5]{DIM}, the authors observe that a Yosida frame is isomorphic to the frame of ideals of an ideally conjunctive distributive join-semilattice.


\begin{proposition}\label{idealconj}
Suppose $L$ is a distributive 0-lattice. The following are equivalent.
\begin{itemize}
\item[(1)] $L$ is ideally conjunctive.
\item[(2)] For all $c \in L$, $\id(c) = \bigcap \{ m \in \Max(L) : c \in m \}$.
\item[(3)] $\Idl L$ is a Yosida frame. 
\end{itemize} 
\end{proposition}

\begin{proof}
(1) $\Leftrightarrow$ (2). This is a special case of \cite[Proposition 4.2]{mz06}, which is itself a special case of \cite[Corollary 3.20]{DIM}.

(2) $\Leftrightarrow$ (3). As discussed in Section \ref{prelim}, $\Idl L$ is an algebraic frame with FIP whose compact elements are precisely the principal ideals of $L$. Since the maximal ideals of $L$ are precisely the maximal elements of $\Idl L$, the desired equivalence is obtained.
\end{proof}

We conclude this section by establishing some relationships between disjunctive, conjunctive, and ideally conjunctive distributive 0-lattices.

\begin{lemma}\label{perplem}
Suppose $L$ is a distributive 0-lattice with $\bigcap \Max(L) = \{ 0 \}$. Then 
\[
\bigcap \{ m \in \Max(L) : a \in m \} \subseteq \id(a)^{\pperp}
\]
 for every $a \in L$.
\end{lemma}

\begin{proof}
Suppose $c \notin \id(a)^{\pperp}$. Then there is $b \in L$ such that $a \wedge b = 0$ and $c \wedge b > 0$. By hypothesis there is $m \in \Max(L)$ such that $c \wedge b \notin m$, and so $c \notin m$. However, $a \in m$ since $a \wedge (c \wedge b) = 0$.  
\end{proof}

\begin{proposition}\label{notop}
Suppose $L$ is a distributive $0$-lattice and $\bigcap \Max(L) = \{ 0 \}$. If $L$ is disjunctive, then $L$ is ideally conjunctive.
\end{proposition}

\begin{proof}
Suppose $L$ is disjunctive and take $a \in L$. By \ref{perplem} and \ref{martframes}, we have \[
\bigcap \{ m \in \Max(L) : a \in m \} \subseteq \id(a)^{\pperp} = \id(a).
\]
 So $L$ is ideally conjunctive by \ref{idealconj}.
\end{proof}

\begin{proposition}\label{disj-conj}
Suppose $L$ is a bounded distributive lattice with $\bigcap \Max(L) = \{ 0 \}$. Then the following are equivalent. 
\begin{itemize}
\item[(1)] $L$ is disjunctive.
\item[(2)] $L$ is conjunctive and $\Max(L) \subseteq \Spec_d(L)$. 
\item[(3)] $L$ is conjunctive with no proper dense elements.
\end{itemize} 
\end{proposition}

\begin{proof}
(1) $\Rightarrow$ (2). If (1) holds, then $L$ is ideally conjunctive by \ref{notop}. But this time $L$ has a top element, so $L$ is conjunctive by \cite[Lemma 2.15]{DIM}. And, since $\Spec(L) = \Spec_d(L)$ by \ref{martframes} and $\Max(L) \subseteq \Spec(L)$ always ($L$ is distributive), one sees that $\Max(L) \subseteq \Spec_d(L)$.
 
 (2) $\Rightarrow$ (3). Suppose (2) holds. Then, since $L$ has 1, we know $L$ is ideally conjunctive by \cite[Lemma 2.15]{DIM}, so $\bigcap \{ m \in \Max{L} : a \in m \} = \idl(a)$ for all $a \in L$ by \ref{idealconj}. Since $\Max(L) \subseteq \Spec_d(L)$ by hypothesis, and since any meet of $d$-elements is a $d$-element, it follows from \ref{martframes} that $\idl(a) = \idl(a)^{\pperp}$. So if $\idl(a)$ is dense, then $\id(a) = \idl(a)^{\pperp} = L$.
 
 (3) $\Rightarrow$ (1). Suppose (3) holds and take $a \in L$. It suffices to show $\id(a)^{\pperp} \subseteq \id(a)$. Take $b \notin \id(a)$. Then, since $L$ is conjunctive, there is $c \in L$ such that $a \vee c < 1$ and $b \vee c = 1$. Since there are no proper dense elements in $L$, there is $0 < d \in L$ such that $d \wedge (a \vee c) = 0$. Now 
\[
d \wedge a \leq (d \wedge a) \vee (d \wedge c) = d \wedge (a \vee c) = 0,
\]
so $d \wedge a = 0$, and similarly $d \wedge c  = 0$, so 
\[
0 < d = d \wedge 1 = d \wedge (b \vee c) = (d \wedge b) \vee (d \wedge c) = (d \wedge b) \vee 0 = d \wedge b.
\]
Thus $b \notin a^{\pperp} = \id(a)^{\pperp}$.
\end{proof}

\begin{remark}
The lattices in \ref{disj-conj} are both disjunctive and conjunctive. Cornish \cite[Theorem 4.2]{cornish} showed that a bounded distributive lattice $L$ is disjunctive and conjunctive if and only if the Dedekind-MacNeille completion of $L$ is a complete Boolean algebra. 
\end{remark}


\section{Characterizations and preservation of $\M$}\label{sectionM}






Recall that $\M$ denotes the class of semi-boolean $\ell$-groups. Accordingly, $G \in \M$ means that $\fC(G)$ is a \Mart frame. We begin with our extension of Bigard's characterization of $\M$.


\begin{theorem}\label{Main}
Suppose $G$ is an $\ell$-group. Then the following statements are equivalent:
\begin{enumerate}[label={\rm \arabic*.}, nolistsep]
\item[(1)]  $\mathbf{C}(G)$ is disjunctive.
\item[(2)] $G \in \M$, that is, $\fC(G)$ is a \Mart frame. 
\item[(3)] $\Spec_d(G) =\Spec(G)$. 
\item[(4)] For all $a \in G$, $G(a)=a^{\pperp}$. 
\item[(5)]   For all $a \in G$, $G(a)$ is a $d$-element of  $\fC(G)$. 
\item[(6)] The set $\Min(G)$ is patch-dense in $\Spec(G)$. 
\item[(7)] The set $\{U(a):  0 < a  \in G \}$ is a $\pi$-base for the patch topology on $\Spec(G)$
\item[(8)] Distinct compact open subsets of $\Spec(G)$ have distinct closures (with respect to the hull-kernel topology).
\item[(9)] For every $a \in G$, $G(a) \in \M$.
\item[(10)] For every $H \in \fC(G)$, $H \in \M$.
\item[(11)] For all $a \in G$, $G(a) = \bigcap \{ p \in \Min(G) : a \in p \}$.
\item[(12)] For all $a,b\in G$, $G(a) = G(b)$ if and only if $a^{\pperp}=b^{\pperp}$.
\end{enumerate}
\end{theorem}

\begin{proof}
 Apply \ref{martframes} with $L \coloneqq \mathbf{C}(G)$ (recall that then $\Idl L$ and $\fC(G)$ are isomorphic as frames). 
\end{proof}


Here are two immediate consequences of \ref{Main}.

\begin{proposition}\label{wu->su}
\begin{itemize}
\item[(1)] If $G$ is a totally-ordered group, then $G \in \M$ if and only if $G \in \Arch$.
\item[(2)] If $G \in \M$, then every weak unit is a strong unit in $G$.
\end{itemize}
\end{proposition}

An $\ell$-group $G$ is called hyperarchimedean if every image of $G$ under an $\ell$-group homomorphism is archimedean (equivalently: $G = G(a) \oplus a^{\perp}$ for every $a \in G$, see \cite[Theorem 55.1]{darnel}); and $G$ is called projectable if $G = a^{\pperp} \oplus a^{\perp}$ for every $a \in G$. Another consequence of \ref{Main} is the following observation of Bigard \cite[Th\'{e}or\`{e}m 4]{bigard}:

\begin{proposition}\label{bigard}
The $\ell$-group $G$ is hyperarchimedean if and only if $G$ is a projectable  and $G\in \M$.
\end{proposition}

 We now supply the reader with many further examples of $G \in \M$.

\begin{example}\label{nonarch}
Let $G$ be any hyperarchimedean $\ell$-group without strong order units. Consider the lex-extension of $G$ by $\Z$ (or any subgroup of $\R$). Recall that the order in $H=\overleftarrow{G\times \Z}$ is defined by $(a,n)\leq (b,m)$ if either $n<m$ or $n=m$ and $a\leq b$. The $\ell$-group $H \notin \Arch$. However, since we have only adjoined a strong unit it follows that $H \in \M$.
On the other hand, if $G$ does have a strong unit, say $u\in G$, then $(u,0)$ is not a strong unit of $H$ but $(u,0)^{\pperp}=H$ so that $H \notin \M$ by \ref{wu->su}. More generally, using results from \cite{wynne}, one can show that if $G, H \in (M)$ and $G \times_{\pi} H$ is an upper extension such that $\pi(G(g))$ has no weak unit for every $g > 0$ in $G$, then $G \times_{\pi} H \in \M$.
\end{example}

\begin{example}
Every existentially closed (e.c.) abelian $\ell$-group is in $\M$. Indeed, building on work of Glass-Pierce and Saracino-Wood, Weispfenning \cite[Theorem 3.6.3(2)]{W} showed that an abelian $\ell$-group is e.c. if and only if it is divisible, has the ``splitting property", has no basic elements, has no weak units, and has the ``disjunction property" (this last property says that $\mathbf{C}(G)$ is disjunctive).
Since a basic result from model theory implies that every abelian $\ell$-group can be embedded in an e.c. abelian $\ell$-group, it follows that every abelian $\ell$-group has an extension in $\M$. In particular, there are numerous non-archimedean  $\ell$-groups in $\M$. For more about e.c. abelian $\ell$-groups, see \cite{W} and \cite{scowcroft} and the references therein; see \cite{wynne} and \cite{wynne2} for some constructions of ``concrete" non-archimedean e.c. abelian $\ell$-groups.
\end{example}

Next we look at how (M) behaves with regard to some common operations and extensions.

\begin{proposition}\label{products}
$\M$ is preserved under each of the following:
\begin{itemize}
\item[(1)] finite direct products;
\item[(2)] arbitrary direct sums;
\item[(3)] $a$-extensions (e.g., the divisible hull).
\end{itemize}
\end{proposition}

\begin{proof}
(1). This is a special case of (2).

(2). Suppose $G = \bigoplus_{i \in I} G_i$ with $G_i \in \M$ ($i \in I$). We show that $G$ obeys condition (4) of \ref{Main}. Suppose $0 \leq a , b \in G$ and $b \notin G(a)$. Let $b = (b_i)$ and $a = (a_i)$. If $b_i \in G_i(a_i)$ for each nonzero coordinate $b_i$ of $b$, then $b \in G(a)$, contrary to our supposition. So $b_i \notin G_i(a_i)$ for some $b_i \neq 0$. Since $G_i \in \M$, it follows that there is $c_i \in G_i$ such that $a_i \wedge c_i = 0$ and $b_i \wedge c_i > 0$.  Let $c \in G$ have $c_i$ as its $i$th coordinate and zero in every other coordinate.  Then $a \wedge c= 0$ and $b \wedge k > 0$, so $b \notin a^{\perp \perp}$. Hence $G \in \M$.

(3). Recall that an $\ell$-group embedding of $G$ into $H$ is said to be an {\it $a$-extension} if for every $b\in H$ there is a $a\in G$ such that $H(b)=H(a)$. The extension is an $a$-extension if and only if the contraction map $S \mapsto S \cap G$ from $\fC(H)$ to $\fC(G)$ is a lattice isomorphism (see \cite{darnel} for more about $a$-extensions). It follows that $a$-extensions preserve (M).
\end{proof}

\begin{remark}
No product of infinitely many non-trivial members of $\M$ is in $\M$. Indeed, suppose that $\{G_i\}_{i\in I}$ is a collection of non-trivial $\ell$-groups such that each $G_i \in \M$ ($i \in I$) and $I$ is infinite. Choose an infinite sequence in $I$, say $\{ i_n \}_{n \in \mathbb{N}}$. For each $n \in \mathbb{N}$, let $0<a_n\in G_{i_n}$ and consider $b_1, b_2$ in $G = \prod_{i \in I} G_i$ given by:
$$
b_1(i)=\begin{cases}
           a_n, & \mbox{if } i=i_n \\
           0, & \mbox{otherwise}.
         \end{cases}
\hbox{ and } \hspace{.1in}
b_2(i)=\begin{cases}
           na_n, & \mbox{if } i=i_n \\
           0, & \mbox{otherwise}.
         \end{cases}
$$
The reader can check that $b_2 \in b_1^{\pperp} \setminus G(b_1)$ in $G$, so condition (4) of \ref{Main} fails in $G$.
\end{remark}

\begin{remark}
Many types of extensions do  not preserve (M). For example, we show that rigid extensions and the lateral completion can fail to transfer (M). Recall that an extension $G\leq H$ is said to be {\it rigid} if for each $b\in H$ there is some $a\in G$ such that $b^{\pperp_H} = a^{\pperp_H}$.  Every $a$-extension is rigid, but the converse fails.  

Consider $C(X,\Z)$, the $\ell$-subgroup of $C(X)$ consisting of the integer-valued functions. Assume that $X$ is zero-dimensional, i.e., $X$ has a base of clopen sets. Straightforward arguments show that $C(X,\Z)$ is always projectable, and that $X$ is pseudo-compact if and only if $C(X,\Z)$ is hyperarchimedean. Therefore, by \ref{bigard}, we have that $C(X,\Z) \in \M$ if and only if $X$ is pseudo-compact and zero-dimensional. Now $C^*(X,\Z)$, the bounded functions in $C(X,\Z)$, is isomorphic to $C(\beta_0X,\Z)$, where $\beta_0 X$ denotes the Banaschewski compactification of $X$ (\cite{bana55}).
Thus, for any non-pseudo-compact zero-dimensional $X$, one has $C^*(X,\Z) \in \M$ but the rigid extension $C(X,\Z) \notin \M$.

Regarding the lateral completion, if $\beta \N$ is the \v{C}ech-Stone compactification of $\N$, then $C(\beta \N, \Z) \in \M$, but its lateral completion is $D(\beta \N,\Z)$, which is not in $\M$ because it has weak units that are not strong. 
\end{remark}

A class $\cR$ of $\ell$-groups is called a radical class if $\cR$ is nonempty and closed under convex $\ell$-subgroups, under isomorphic images, and under joins of convex sub-$\ell$-groups (see \cite[Definition 36.1]{darnel}). For example, the class  $\A$ of abelian $\ell$-groups and the class $\Arch$ are both radical classes.

Originally we proved that $\M \cap \A$ is a radical class. We thank the anonymous reviewer for showing us how to use \ref{reflemma} to simplify the proof and remove the abelian hypothesis.

\begin{theorem}\label{rad}
$\M$ is a radical class.
\end{theorem}

\begin{proof}
Since $\M$ is closed under convex $\ell$-subgroups by  \ref{Main}, and is obviously closed under isomorphic copies, it remains to check that $\M$ is closed under joins of convex $\ell$-subgroups.

First, note that if $a,b \geq 0$ in $G$, and if $G(a)$ and $G(b)$ are in $\M$, then  $G(a) \vee G(b) = G(a \vee b) \in \M$ by \ref{reflemma}. Indeed, if $L \coloneqq \mathbf{C}(G(a \vee b))$, then $L$ has top $G(a \vee b)$, bottom $ \{0 \}$, and $u_0 = G(a)$ and $u_1=G(b)$ with $u_0 \vee u_1 = G(a \vee b)$. And $\idl(u_0) = \mathbf{C}(G(a))$ and $\idl(u_1) = \mathbf{C}(G(b))$ are disjunctive because $G(a)$ and $G(b)$ are in $\M$ by hypothesis. 

So, by induction, $\bigvee_{i=1}^n G(a_i) \in \M$ if  $G(a_i) \in \M$ for each $i \in \{ 1, \dots, n \}$. 

Finally, suppose $G$ is a (possibly non-abelian) $\ell$-group with convex $\ell$-subgroups $A_{\lambda}$, for $\lambda \in \Lambda$, such that $A_{\lambda} \in \M$ for every $\lambda \in \Lambda$. Were $\bigvee_{\lambda \in \Lambda} A_{\lambda} \notin \M$, then some finite join $\bigvee_{i=1}^n G(a_i)$, with $a_i \in A_{\lambda_i}$ for $i \in \{1, \dots, n \}$, would fail to be in $\M$, but $G(a_i) \in \M$ for each $i \in \{ 1, \dots, n \}$ by \ref{Main} since $A_{\lambda_i} \in \M$ by hypothesis. Thus, $\bigvee_{\lambda \in \Lambda} A_{\lambda} \in \M$.
\end{proof}

\begin{remark}
\ref{rad} raises the question: Are there any non-abelian $G \in \M$? We do not know.
\end{remark}


\section{$\Y$ and $\M \cap \Arch$}\label{sectionY}

Recall that $G \in \Y$ means that $\fC(G)$ is a Yosida frame. In this section we analyze $\Y$ and compare it with $\M \cap \Arch$.



An $\ell$-group $G$, even if it is a $\bf{W}$-object, might  have $\Max(G) = \emptyset$ (e.g., $G=D(X)$, where $X$ is the absolute of the unit interval $[0,1]$).
We say that $G$ has  enough maximal convex $\ell$-subgroups (emc, for the sake of abbreviation) if $\bigcap \Max(G) = \{ 0 \}$ (note: if $G \neq \{ 0 \}$, this implies that $\Max(G) \neq \emptyset$). Observe that $G$ having emc is equivalent to saying that $G$ is embeddable in a product of copies of $\R$. Such $\ell$-groups are obviously archimedean and, in fact, they form a proper subclass of $\Arch$; see \ref{MnotY} below for a $G \in \M \cap \Arch$ without emc. Since $\{0\}$ is a principal convex $\ell$-subgroup, if $G \in \Y$, then $G$ has emc, and so $\Y \subseteq \Arch$. We can say more: see Theorem \ref{Yosida} below.

\begin{notation}\label{inducedemb}
Suppose $\cM \subseteq \Max(G)$ with $\bigcap \cM = \{ 0 \}$. Then, by H\"{o}lder's theorem, $G/m$ is isomorphic to a subgroup of $\mathbb{R}$ for each $m \in \cM$, and one may view the canonical surjection $\pi_m : G \to G/m$ as an $\ell$-homomorphism from $G$ into $\R$. Since $\bigcap \cM = \{ 0 \}$, we thus have the embedding $\psi \colon G \to \mathbb{R}^{\cM}$, given by $\psi(a) = (\pi_m(a))_{m \in \cM}$, which we call \emph{the embedding induced by $\cM$}. In that situation, for $a \in G$, we write $\sC_{\cM}(a)$,  $\sZ_{\cM}(a)$, and $I(\sZ_{\cM}(a))$ for $U(a) \cap \cM$, $V(a) \cap \cM$, and $\{ b \in G : \sZ_{\cM}(b) \supseteq \sZ_{\cM}(a) \}$, respectively. This fits with our notation from \ref{zsetnotation}, though here we replace the subscript ``$\psi$" by ``$\cM$" since $\cM$ determines $\psi$.
\end{notation}

The following, which is our main result concerning $\Y$, extends \cite[Proposition 5.2]{mz06}. 

\begin{theorem}\label{Yosida}
Let $G$ be an $\ell$-group. Then the following are equivalent.
\begin{enumerate}[label={\rm \arabic*.}, nolistsep]
\item[(1)] $\mathbf{C}(G)$ is ideally conjunctive.
\item[(2)] $G\in \Y$, that is, $\fC(G)$ is a Yosida frame.
\item[(3)] There is an index set $X$ and an embedding $\varphi \colon G \to \R^X$ such that $G(a) = I(\sZ_{\varphi}(a))$ for every $a \in G$.
\item[(4)] There is $\cM \subseteq \Max(G)$ such that $\bigcap \cM = \{ 0 \}$ and $G(a) = I(\sZ_{\cM}(a))$ for every $a \in G$.
\item[(5)] There is $\cM \subseteq \Max(G)$ such that $\bigcap \cM = \{ 0 \}$ and $\cM$ is patch dense in $\Spec(G)$.
\item[(6)] $H \in \Y$ for every $\ell$-subgroup $H$ of $G$.
\item[(7)] $G/G(a)$ has emc for each $a\in G$.
\end{enumerate}
\end{theorem}

\begin{proof}
(1) $\Leftrightarrow$ (2). This follows from \ref{idealconj}.

(2) $\Leftrightarrow$ (3). This is just a rephrasing of \cite[Proposition 5.2]{mz06}.

(3) $\Leftrightarrow$ (4). That (4) implies (3) is immediate, just take $X=\cM$. Suppose (3) holds.  For each $x \in X$, let  $m_x = \varphi^{-1}(\{ f \in \R^X : f(x) = 0 \})$. One sees there are exactly two possibilities for each $m_x$: either $m_x = G$ or $m_x \in \Max(G)$. Let $\cM = \{ m_x : m_x \in \Max(G) \}$. Then clearly $\cM \subseteq \Max(G)$. Take $a \in \bigcap \cM$. Then $\varphi(a)(x) = 0$ for every $x \in X$, so $a = 0$ in $G$ because $\varphi$ is an embedding. Thus $\bigcap \cM = \{ 0 \}$. It remains to check that $G(a) = I(\sZ_{\cM}(a))$ for every $a \in G$.

Suppose $a \in G$ and $b \in I(\sZ_{\cM}(a)))$. Since $G(a) = I(\sZ_{\varphi}(a)))$ by hypothesis,  to show that $b \in G(a)$, which is all that is required, it suffices to show that $\sZ_{\varphi}(b) \supseteq \sZ_{\varphi}(a)$. Take $x \in \sZ_{\varphi}(a)$. Then $a \in m_x$. If $m_x = G$, then $b \in m_x$, so $x \in \sZ_{\varphi}(b)$. If $m_x \neq G$, then $m_x \in \cM$ so $m_x \in \sZ_{\cM}(a)$. Since $b \in I(\sZ_{\cM}(a))$ by hypothesis, it follows that $m_x \in \sZ_{\cM}(b)$, which means $b \in m_x$. Hence  $ x \in \sZ_{\varphi}(b)$.

(4) $\Rightarrow$ (5). Suppose $\cM \subseteq \Max(G)$ is such that $\bigcap \cM = \{ 0 \}$ and $G(a) = I(\sZ_{\cM}(a))$ for every $a \in G$. We  show that $\cM$ is patch dense in $\Spec(G)$. Suppose  $U(b) \cap V(a)$ is a nonempty basic open set of $\Spec(G)$. Since $U(b) \cap V(a) \neq \emptyset$, we know $b \notin G(a)$. Since $G(a) = I(\sZ_{\cM}(a))$, it follows that there is $m \in \cM$ such that $m \in \sZ_{\cM}(a)  \ssm \sZ_{\cM}(b)$. Thus $b \notin m$ and $a \in m$, so $m \in U(b) \cap V(a)$. Hence $\cM$ is patch dense in $\Spec(G)$.

(5) $\Rightarrow$ (4). Suppose $\cM \subseteq \Max(G)$ is such that $\bigcap \cM = \{ 0 \}$ and $\cM$ is patch dense in $\Spec(G)$. Take $a,b \in G$ such that $b \notin G(a)$. To show that $G$ obeys (3), it suffices to show that $b \notin I(\sZ_{\cM}(a))$.
 Consider the basic open set $U(b) \cap V(a)$. Since $b \notin G(a)$, the set $U(b) \cap V(a) \neq \emptyset$, so by hypothesis there is $m \in \cM$ with $m \in U(b) \cap V(a)$. It follows that $m \in \sZ_{\cM}(a) \ssm \sZ_{\cM}(b)$. Hence $b \notin I(\sZ_{\cM}(a))$.

(6) $\Rightarrow$ (1). Immediate.

(3) $\Rightarrow$ (6).  Suppose $G$ obeys (2) for some $X$ and $\varphi$. Take $H \leq G$ and $a \in H$, and let $\varphi'$ be the restriction of $\varphi$ to $H$. Then $\varphi'$ is an embedding and $b \in I(\sZ_{\varphi'}(a)) \subseteq I(\sZ_{\varphi}(a)) = G(a)$ implies $b \in H(a)$. Thus (2) holds for $H$, which implies $H \in \Y$.

(2) $\Leftrightarrow$ (7). Fix $0 < a \in G$, and let $\rho \colon G \to G/G(a)$ be the canonical surjection. For any $H \in \fC(G)$, if $G(a) \subseteq H$, then $\rho^{-1}(\rho(H)) = H$, and $H \in \Max(G)$ if and only if $\rho(H) \in \Max(G/G(a))$. Moreover, for any $\{ H_i \}_{i \in I} \subseteq \fC(G)$, one has $G(a) = \bigcap_{i \in I} H_i$ if and only if $\{ 0  \} = \bigcap_{i \in I} \rho(H_i)$. The desired equivalence follows.
\end{proof}

\begin{remark}
Note that \ref{Yosida}(6) implies that $G \in \Y$ if and only if $\Max(G)$ is patch dense in $\Spec(G)$.
\end{remark}

\begin{example}
Conrad (\cite[Theorem 1.1]{conrad74}) showed that an $\ell$-group is hyperarchimedean if and only if there is an index set $X$ and an embedding $\varphi \colon G \to \R^X$ such that ($\ast$) for any $0<a,b \in G$ there is an $n\in \N$ such that $n\varphi(a)(x)> \varphi(b)(x)$ for all $x\in \sC_{\varphi}(a)$. Such an embedding satisfies (2) of Theorem \ref{Yosida}: for any $0 < a \in G$, if $b \in I(\sZ_{\varphi}(a))$, then $\sZ_{\varphi}(b) \supseteq \sZ_{\varphi}(a)$, which implies $\sC_{\varphi}(b) \subseteq  \sC_{\varphi}(a)$, and so $b \in G(a)$ by ($\ast$). Therefore, every hyperarchimedean $\ell$-group is in $\Y$.
\end{example}

The following two results are easy consequences of \ref{Yosida}.

\begin{proposition}
If $G$ is a totally-ordered group, then $G \in \Y$ if and only if $G \in \Arch$.
\end{proposition}

\begin{proposition}
$\Y$ is preserved by each of the following:
\begin{itemize}
\item[(1)] finite direct products;
\item[(2)] arbitrary direct sums
\item[(3)] $a$-extensions.
\end{itemize}
\end{proposition}

\begin{remark}
$\Y$ is not closed under arbitrary products, e.g., $\R \in (Y)$ but any an infinite direct product of copies of $\R$ is not in $\Y$ because \ref{Yosida}(5) fails. Since a totally-ordered group is in $\Y$ if and only if it is in $\Arch$, it follows that $\Y$ is not closed under homomorphic images: any $G \in \Y$ with a non-maximal prime subgroup will have a homomorphic image that is not in $\Arch$. 
\end{remark}

\begin{example}\label{MZex}
This example, which is taken from \cite[Example 5.5]{mz06}, shows  two things: (i) that $\Y \nsubseteq \M$ and (ii) that there can be $\cP \subseteq \Max(G)$ with $\bigcap \cP = \{ 0 \}$ yet $\cP$ is not patch dense in $\Max(G)$. Let $(0,\infty)$ be the set of positive real numbers and let $G$ be the $\ell$-subgroup of $C((0,\infty))$ consisting of all piecewise linear functions. Clearly $G \in \Arch$. One checks that $\Max(G) = \{ M_t : t \in (0,\infty) \} \cup \{ M_0, M_{\infty} \}$, where
\begin{itemize}
\item $M_t = \{ f \in G : f(t) = 0 \}$ for $t \in (0,\infty)$;
\item $M_0 = \{ f \in G : \lim_{t \to 0} f(t) = 0 \}$;
\item $M_{\infty} = \{ f \in G : f \mbox{ is bounded} \}$.
\end{itemize}
By verifying that $\mathcal{M} = \Max(G)$ satisfies condition (4) of \ref{Yosida}, one may show that $G \in \Y$. Note that $\Max(G) \ssm \{ M_{\infty} \}$ is a subset of $\Max(G)$ with trivial intersection which is not patch dense in $\Max(G)$.
Note too that $G \notin \M$ since the constant function with value 1 is a weak unit that is not strong.
\end{example}

A characterization of $\M \cap \Arch$, like that of Theorem \ref{Yosida} for $\Y$, has so far eluded us, though the next result relates $\Y$ and $\M$ under the additional assumption of emc.

\begin{proposition}\label{MartYos}
Suppose $G$ has emc. If $G \in \M$, then $G \in \Y$.
\end{proposition}

\begin{proof}
Suppose $G$ has emc and $G \in \M$. Then $L \coloneqq \mathbf{C}(G)$ is disjunctive and $\bigcap \Max(L) = \{ 0 \}$. So $L$ is ideally conjunctive by \ref{notop}, and $G \in \Y$ by \ref{Yosida}.
\end{proof}


If $G \in \M$ and $G \notin \Arch$, then $G \notin \Y$, so $\M \nsubseteq \Y$. The following example shows that even $\M \cap \Arch \nsubseteq \Y$.

\begin{example}\label{MnotY}
We construct $G \in \M \cap \Arch$ with $G \notin \Y$. 

Let $K$ be the Cantor set, viewed as a subset of the unit interval $[0,1] \subseteq \R$ as usual, let $K^* = K \cup \{ -1 \}$, and let $\overline{\Z} = \mathbb{Z} \cup \{ \pm \infty \}$ have the usual order topology.  Our intended example $G$ will be an $\ell$-group contained in the subset of the Cartesian  product $C(K,\overline{\Z})^{K^*}$ consisting of those $f = (f_x)_{x \in K^*}$ such that $f_x = 0$ for all but finitely many $x \in K^*$. What makes this somewhat awkward is that $C(K,\overline{\Z})$ is not a group, so neither is $C(K,\overline{\Z})^{K^*}$, but of course various subsets are groups. In particular, note that $C(K,\Z)$ is an $\ell$-group in $C(K,\overline{\Z})$ and that functions in $C(K,\Z)$ take only finitely many values (hence the pre-images of those values are disjoint clopen sets in $K$).

First, we describe a particular collection $\{ a^k \}_{k \in K}$ of unbounded functions in $C(K, \overline{\Z})$. For $k \in K$, choose a disjoint collection $\{ U_{k,n} \}_{n \in \N} \subseteq \Clop(K)$ such that $k \notin U_{k,n} \neq \emptyset$ for every $n \in \N$ and $K = \{ k \} \cup \bigcup_{n \in \N} U_{k,n}$; this is possible because each singleton set in $K$ is a zero-set and $K$ is compact and zero-dimensional and every cozero-set of $K$ is Lindel\"{o}f.
Now let $a^k \colon K \to \overline{\Z}$ be given by
\[
a^k(y) = \left \{
\begin{array}{ll}
+\infty & \mbox{ if $y=k$,} \\
n & \mbox{ if $y \in U_{k,n}$.}
\end{array}
\right.
\]
Note that $a^k  \in C(K,\overline{\Z})$, and that  if $k'\neq k$ in $K$, then there is $U \in \Clop(K)$ such that $k' \in U$, $k \notin U$, and $a^k$ is constant on $U$. Moreover, if $b \in C(K,\Z)$, and if $\{k_i \}_{i=1}^n$ is a finite subset of $K$, then $b + \sum_{i=1}^n m_ia^{k_i}$ is a well-defined element of $C(K,\overline{\Z})$ for any $\{ m_i \}_{i=1}^n \subseteq \Z$. Indeed, for each $i \in \{1,\dots, n \}$, there is $U_i \in \Clop(K)$ such that $k_i \in U_i$, $k_j \notin U_i$ when $j \neq i$, and $b$ and each $a^{k_j}$ ($j \neq i$) are constant on $U_i$. So there is $c_i \in C(K,\Z)$ such that $c_i$ is constant on $U_i$ and vanishes on $K-U_i$, and $b+\sum_{i=1}^n m_ia^{k_i}$ agrees with $c_i+m_ia^{k_i}$ on $U_i$; the latter is clearly a well-defined element of $C(K,\overline{\Z})$. By shrinking $U_i$ if necessary, one may assume that $c_i+m_ia^{k_i}$ does not change sign on $U_i$. It follows that $b + \sum_{i=1}^n m_ia^{k_i}$ may be written as a sum of disjoint terms, each of which is either bounded or comparable with 0. These observations can help one to check that our intended example $G$, to be described shortly, is an $\ell$-group. 


Next, we define two special collections, namely $\{ f^k \}_{k \in K}$ and $B$, that will generate our intended example $G$.
For each $k \in K$, let $f^k = (f^k_x)_{x \in K^*}$ be the element of $C(K,\overline{\Z})^{K^*}$ given by
 \[
f^k_x = \left \{
\begin{array}{ll}
a^k & \mbox{ if  $x = -1$ or $x = k$;} \\
0 & \mbox{ otherwise.} 
\end{array}
\right.
\]
Let $B$ consist of those $b = (b_x)_{x \in K^*}$ in  $C(K,\overline{\Z})^{K^*}$ such that
\begin{itemize}
\item  $b_x = 0$ for all but finitely many $x \in K^*$;
\item $b_x \in C(K,\Z)$ for every $x \in K^*$;
\item $b_k(k) = 0$ for every $k \in K$.
\end{itemize}
In particular, note that for every $k \in K$, if $b \in B$, then $b_k$ vanishes on some  $U \in \Clop(K)$ with $k \in U$ because $b_k \in \C(K,\Z)$ and $b_k(k)=0$ by hypothesis. Observe  too that $B$ is an $\ell$-subgroup of the direct sum $\bigoplus_{x \in K^*} C(K,\Z)$.

Finally, let $g \in G$ mean that $g = b + \sum_{i=1}^n m_if^{k_i}$, where $b \in B$, $\{ k_i \}_{i=1}^n$ is a finite subset of $K$, and $\{ m_i \}_{i=1}^n \subseteq \Z$. 
Using the observations about the functions $a^k$ made above, it can be shown that $G$ is an $\ell$-group in $C(K,\overline{\Z})^{K^*}$. Note that if $g \in G$, and if $g_x(k) \neq \pm \infty$ for some $x \in K^*$ and $k \in K$, then $g_x$ is finite and constant on some clopen set of $K$ containing $k$. Moreover, note that if $g_{-1}(k) \neq \pm \infty$ for some $k \in K$, then $g_{k}$ vanishes on some clopen set in $K$ that contains $k$.

To see that  $G \in \Arch$, suppose that $0 < g \leq g'$ in $G$. Since $g > 0$, there are $x \in K^*$ and $k \in K$ such that $g_x(k) > 0$. Since $g$ and $g'$ can take the value $+\infty$ at only finitely many points, it follows that there is $y \in K$ such that $0 < g_x(y) \leq g'_x(y) < +\infty$. Thus, there is $n \in \N$ such that $ng \nleq g'$. Hence $G \in \Arch$.

Let $B_{-1}$ consist of those $b = (b_x)_{x \in K^*} \in B$ such that $b_k = 0$ for all  $k \in K$. To show that $G \notin \Y$, it suffices to show $B_{-1} \subseteq \bigcap \Max(G)$. Take $M \in \Max(G)$. Towards a contradiction, assume there is $k_0 \in K$ such that $f_{-1}(k_0) = 0$ for every $f \in M$. Then the convex $\ell$-subgroup $M'$ of $G$ generated by $M \cup \{ b \}$, where $b$ is any element of $B_{-1}$  with $b_{-1}(k_0) \neq 0$, does not contain the function $f^{k_0}$ (because $f^{k_0}_{-1}(k_0) = +\infty$), which contradicts the maximality of $M$. So for each $k \in K$, there are $0 < e^k \in M$ and $U_k \in \Clop(K)$ such that $k \in U$ and $e^k_{-1}$ does not vanish on $U_k$. Since $K$ is compact, there is a finite set $\{ k_i \}_{i \in I}$ such that $\{ U_{k_i} \}_{i \in I}$ covers $K$. Since every element $0 \leq b \in  B_{-1}$ lies below some integer multiple of $\bigvee_{i \in I} e^{k_i} \in M$, it follows that $B_{-1} \subseteq M$. Hence $B_{-1} \subseteq \bigcap \Max(G)$ as claimed.

It remains to check that $G \in \M$.  Take $0 < g,h \in G$ and suppose $g \notin G(h)$. If we can show that $g \notin h^{\perp \perp}$, then $h^{\perp \perp} = G(h)$ and $G \in \M$ by \ref{Main}. Since $g \notin G(h)$, it can be shown that there is $k \in K$ such that $g_{-1}(k) = +\infty > h_{-1}(k)\geq 0$. It follows that there is a clopen set $U$ in $K$ such that $k \in U$ and $h_k(u) = 0 < g_k(u)$ for every $u \in U$. Since $K$ has no isolated points, one may choose a nonempty clopen set $V \subseteq U$ such that $k \notin V$. Let $b \in B$ be such that $b_k$ is the characteristic function of $V$ and $b_x=0$ for $k \neq x \in K^*$. Then $b \in G$ and $b\wedge h=0 < b \wedge g$. Hence $g \notin h^{\pperp}$. 
\end{example}

We turn now to {\bf W}-objects. Here is our main result concerning this class.

\begin{theorem}\label{main}
Let $(G,u)\in {\bf W}$. The following statements are equivalent.
\begin{enumerate}[label={\rm \arabic*.}, nolistsep]
\item[(1)] $G \in \M$. 
\item[(2)] $(G,u)\in {\bf W}^*$ and the Yosida embedding $G \leq C(YG)$ has the property that for any $a,b \geq 0$, if $\coz(b) \subseteq \Cl \coz(a)$, then $b\in G(a)$.
\item[(3)] $G \in \Y$ and $\Max(G)  \subseteq \Spec_d(G)$. 
\item[(4)] $G \in \Y$ and every weak order unit of $G$ is a strong order unit. 
\end{enumerate}
\end{theorem}

\begin{proof}
If $(G,u) \in \bf{W}^*$, then $F \coloneqq \mathbf{C}(G)$ is a bounded distributive lattice, and $\bigcap \Max(L) = \{ 0 \}$ follows from \ref{YRT}. So in this case (1), (3), and (4) are equivalent by \ref{disj-conj}. Now each of (1)--(4) implies that $(G,u) \in \bf{W}^*$. Indeed, for (2) and (4) this is immediate. If $G \in \M$, then $G(u) = u^{\pperp} = G$, so $u$ is strong. If (3) holds, then again $G(u) = u^{\pperp} = G$, this time because $G(u)$ is an intersection of maximal convex $\ell$-subgroups that are $d$-elements in $\fC(G)$ (any meet of $d$-elements if a $d$-element).

The condition on the Yosida embedding in (2) is just a rephrasing of $a^{\pperp} \subseteq G(a)$ using \ref{polarYG}, so (2) is clearly equivalent to (1). 
\end{proof}

\begin{remark}
The example $G$ in \ref{MnotY} may be used to show that $\Y$ is not a radical class. Indeed, if $G(a)$ is any principal convex $\ell$-subgroup of $G$, then  $G(a) \in \M$ by \ref{Main} because $G \in \M$, and therefore $G(a) \in \Y$ by \ref{main} because $G(a) \in \bf{W}$. However, $G$ is the join of all its principal convex $\ell$-subgroups and $G \notin \Y$. 
\end{remark}

\begin{remark}
If $G \in \bf{W}^*$, then one has $G(a) \subseteq I(Z(a)) \subseteq a^{\pperp}$ for each $a \in G$ (here $Z(a)$ is with respect to the Yosida embedding). In that case, $G \in \M$ says $G(a) = a^{\pperp}$ and $G \in \Y$ says $G(a) = I(Z(a))$, i.e., $\M$ and $\Y$ may described in terms of coincidence of types of ideals (here ``ideal" is just short for ``convex $\ell$-subgroup" since the group is abelian). For further analysis of $\M$ and $\Y$ in $\bf{W}^*$ along these lines, including construction of more examples, see our companion paper \cite{bhmw}.
\end{remark}


\section{The $G+B$ Construction}\label{G+B}

In this section we modify a construction from the theory of commutative rings, known as ``the $A+B$ construction" (see \cite[Chapter VI]{huckaba}), to make it useful in the context of $\ell$-groups.

Let $G$ be an abelian $\ell$-group and let $P \subseteq \Spec(G)$ with $\bigcap P = \{0\}$. For each $(p,n) \in P \times \N$, let $G(p,n)$ be a copy of $G/p$, and let $G[P,\N] =\Pi \{ G(p,n) : (p,n) \in P \times \N \}$. For each $g \in G$, let $\hat{g} \in G[P,\N]$ be given by $\widehat{g}(p,n) = g + p$. Since $\bigcap P = \{ 0 \}$, we have, as in \ref{inducedemb}, the induced embedding $g \mapsto \widehat{g}$ of $G$ into $G[P,\N]$; we denote the image of $G$ by $\widehat{G}$. 

Next, for each $g \in G$ and each $(p,n) \in P \times \N$, we define $g_{(p,n)} \in G[P,\N]$ by
\[
g_{(p,n)}(q,m) = \left \{
\begin{array}{ll}
\hat{g}(p,n) & \mbox{ if } (q,m) = (p,n); \\
0 & \mbox{ otherwise.}
\end{array}
\right.
\]
Let $B$ be the $\ell$-subgroup of $G[P,\N]$ generated by $\{ g_{(p,n)} : g \in G, (p,n) \in P \times \N \}$; note that $B$ is simply the direct sum $\bigoplus \{ G/p : (p,n) \in P \times \N \}$. The $\ell$-subgroup of $G[P,\N]$ generated by $\widehat{G}$ and $B$ is denoted by $\widehat{G}+B$. The $\ell$-group $\widehat{G}+B$ is what we call ``the $G+B$ construction".

\begin{observations}\label{obs}
\begin{itemize}
\item[(i)] Every element of $\widehat{G}+B$ can be written uniquely in the form $\widehat{g}+b$, with $\widehat{g} \in \widehat{G}$ and $b \in B$.
\item[(ii)] $B$ is an $\ell$-ideal in $\widehat{G} +B$ and $(\widehat{G}+B)/B \cong G$ (the latter follows from the Second Isomorphism Theorem \cite[Theorem 8.6(b)]{darnel}); the associated canonical surjection will be denoted by $\beta$.
\item[(iii)] It should be apparent that the properties of $\widehat{G}+B$ depend on the choice of $P$, even though the notation $\widehat{G}+B$ does not indicate that explicitly.
\end{itemize}
 \end{observations}


Our first result characterizes $\Spec(\hat{G}+B)$ and $\Min(\widehat{G}+B)$. For each $(p,n) \in P \times \N$, let $m_{(p,n)} = \{ f \in \widehat{G}+B : f(p,n) = 0 \}$. Then $m_{(p,n)}$ is the kernel of the projection map from $\widehat{G}+B$ onto $G(p,n)$, so $m_{(p,n)} \in \Spec(\widehat{G}+B)$ because $G(p,n)$ is totally ordered.



\begin{proposition}\label{spec(G+B)}
\begin{enumerate}
\item $I \in \Spec(\widehat{G}+B)$ if and only if
\begin{itemize}
\item $I = \widehat{p}+B = \{ \widehat{g} +b : g \in p, b \in B \}$ for some $p \in \Spec(G)$ or
\item $I = \{ f \in  \widehat{G}+B : f(p,n) \in q/p \}$ for some $(p,n) \in P \times \N$ and $q \in \Spec(G)$ with $p \subseteq q$. 
\end{itemize}
\item $I \in \Min(\widehat{G}+B)$ if and only if 
\begin{itemize}
\item $I = \widehat{p} + B =\{ \widehat{g}+b : g \in p, b \in B \}$ for some $p \in \Min(G)$ or
\item $I = m_{(p,n)}$ for some $(p,n) \in P \times \N$.
\end{itemize}
\end{enumerate}
\end{proposition}

\begin{proof}
(1). By \ref{obs}(ii) and \cite[Proposition 9.11]{darnel}, the map $\beta$ induces a bijection between the prime ideals of $\widehat{G}+B$ containing $B$ and the prime ideals of $G$. Such primes have the form $\widehat{p}+B = \{ \widehat{g}+b : g \in p, b\in B \}$ for $p \in \Spec(G)$.

Suppose $(p,n) \in P \times \N$ and $q \in \Spec(G)$ with $p \subseteq q$. Since the image $q'$ of $q$ under the canonical surjection from $G$ to $G/p$ is a prime subgroup of $G/p$, one sees that  $\{ f \in  \widehat{G}+B : f(p,n) \in q/p \}$, which contains $m_{(p,n)}$, is also a prime subgroup because it is  the kernel of the composite map $\widehat{G} +B \to G(p,n) \to G(p,n)/q'$. It is clear that every prime subgroup of $\widehat{G}+B$ that contains $m_{(p,n)}$ is of this form.

Finally, suppose $I \in \Spec(\widehat{G}+B)$. We consider two cases: either $m_{(p,n)} \subseteq I$ for some $(p,n) \in P \times \N$ or not. In the first case, see the previous paragraph. For the second case, suppose that $m_{(p,n)} \nsubseteq I$ for every $(p,n) \in P \times \N$. Then for each $(p,n) \in P\times \N$ let $h_{(p,n)}\in m_{(p,n)} \ssm I$. For any $g\in G$, $h_{(p,n)}\wedge g_{(p,n)}=0$ and so since $I$ is prime it follows that $g_{(p,n)} \in I$. Consequently, $B\subseteq I$.

(2). Sufficiency is clear. To establish necessity, take $I \in \Min(\widehat{G}+B)$. There are two cases. If $I = \widehat{p}+B$ for $p \in \Spec(G)$, then if $q \in \Spec(G)$ has $q \subseteq q$, one sees that $\hat{q} +B \in \Spec(\widehat{G}+B)$ and $q+p \subseteq I$, so $q+p = I$. It follows that $q = p$ and $p \in \Min(G)$. Now suppose $I = \{ f \in \widehat{G}+B : f(p,n) \in q/p \}$ for some $(p,n) \in P\times \N$ and $p \subseteq q \in \Spec(G)$. Then $m_{(p,n)} \subseteq I$ and $m_{(p,n)} \in \Spec(\widehat{G}+B)$, so $I = m_{(p,n)}$.
\end{proof}

Observe that if $g\in G\ssm \bigcup P$, then $g(p,n)\neq 0$ for all $(p,n) \in P\times \N$ and hence $\widehat{g}$ is a weak order unit of $\widehat{G}+B$ and $g$ is a weak order unit of $G$ (recall that $\bigcap P=\{0\}$). This leads to a characterization of the principal polars of $\widehat{G}+B$. Recalling \ref{zsetnotation}, we write $\sC(f)=\{(p,n)\in P\times \N: f(p,n)\neq 0\}$ and $\sZ(f) =\cI \ssm \sC(f)$ for $f \in \widehat{G}+B$. Observe that the weak order units of $\widehat{G}+B$ are precisely those $f\in \widehat{G}+B$ such that $\sZ(f)=\emptyset$. Moreover, for $g\in G$, $\widehat{g}$ is a weak order unit of $\widehat{G}+B$ if and only if $g\in G\ssm \bigcup P$.

\begin{proposition}\label{polar2}
\begin{itemize}
\item[(1)] If $f,h\in \widehat{G}+B$, then $|f|\wedge |h|=0$ if and only if $\sC(f)\cap \sC(h)=\emptyset$.
\item[(2)] If $f\in \widehat{G}+B$, then $f^{\pperp} = \{h\in \widehat{G}+B: \sC(h)\subseteq \sC(f)\}$.
\end{itemize}
\end{proposition}

\begin{proof}
(1). Straightforward.

(2). Without loss of generality, we assume that $0\leq f$. Let $0\leq h\in \widehat{G}+B$ satisfy $\sC(h)\subseteq \sC(f)$. If $k\in \widehat{G}+B$ with $k\wedge f=0$, then $\sC(k) \cap \sC(f)=\emptyset$ and hence also $\sC(k)\cap \sC(h)=\emptyset$. Thus, $k\wedge h=0$ and so $h\in f^{\pperp}$.

For the reverse direction, suppose  $h\in \widehat{G}+B$ is such that $\sC(H) \nsubseteq \sC(f)$. Then there is some $(p,n) \in \sC(h)\ssm \sC(f)$. Choose $e\in G\ssm p$ and observe that $\sC(e_{(p,n)})\cap \sC(f)=\emptyset$, whence $e_{(p,n)} \wedge f=0$. But $\sC(e_{(p,n)})=\{(p,n) \}\subseteq \sC(h)$ which implies $e_{(p,n)} \wedge h\neq 0$, so $h \notin f^{\pperp}$.
\end{proof}

Next, we characterize $\Spec_d(\widehat{G}+B)$. Clearly, each prime of the form $m_{(p,n)}$ is in $\Spec_d(\widehat{G}+B)$. In fact, one has $m_{(p,n)} \in \Max_d(\widehat{G}+B)$. Indeed, if $I \in \Spec(\widehat{G}+B)$ and $m_{(p,n)} \subsetneq I \neq G$, then by \ref{spec(G+B)} we have $I = \{ f \in \widehat{G}+B : f(p,n) \in q/p \}$ for some $q \in \Spec(G)$ with $p \subsetneq q \subsetneq G$. Choose any $g\in q\ssm p$ and any $h\in G\ssm q$. Then we have $g_{(p,n)} \in I \ssm m_{(p,n)}$ and $g_{(p,n)}^{\pperp} \nsubseteq I$ since $h_{(p,n)} \in g_{(p,n)}^{\pperp}\ssm I$. Thus, $I$ is not a $d$-subgroup of $\widehat{G}+B$. It remains to determine which primes of the form $\hat{p}+B$, with $p \in \Spec(G)$, are in $\Spec_d(G+B)$.

\begin{definition}
For $g\in G$, set $V_P(g)=P\cap V(g)=\{p\in P:g\in p\}$. Let $H\in \fC(G)$. We say that $H$ is a {\it $P$-element} if $g \in H$ whenever $g \in G$, $h\in H$, and $V_P(h)\subseteq V_P(g)$.
\end{definition}

\begin{theorem}\label{Pelement}
Let $q \in \Spec(G)$. The following statements are equivalent:
\begin{enumerate}[label={\rm \arabic*.}, nolistsep]
\item[(1)] $\hat{q}+B \in \Spec_d(\widehat{G}+B)$.
\item[(2)] $q$ is a $P$-element.
\item[(3)] For every $g\in q$, $\bigcap V_P(g)\subseteq q$.
\item[(4)] $q$ belongs to the patch closure of $P$ in $\Spec(G)$.
\end{enumerate}
\end{theorem}

\begin{proof}
(1) $\Rightarrow$ (2). Suppose $\widehat{q}+B \in \Spec_d(G+B)$. Let $g\in q$ and $h\in G$ satisfy $V_P(g)\subseteq V_P(h)$. Let $(p,n) \in \sC(g)$. This means that $g\notin p$, and so $h \notin p$. It follows that $\sC(\widehat{h})\subseteq \sC(\widehat{g})$. Consequently, $\widehat{h} \in \widehat{g} ^{\pperp}$ in $\widehat{G}+B$. So $\widehat{g} \in \widehat{q}+B$ since $\widehat{q}+B \in \Spec_d(G+B)$. Therefore, $g\in q$ and $q$ is a $P$-element.

(2) $\Rightarrow$ (3). Suppose that $q$ is a $P$-element and let $g\in q$. For any $h\in \bigcap V_P(g)$, it follows that $V_P(g)\subseteq V_\cP(h)$, whence $h\in q$.

(3) $\Rightarrow$ (4). Let $g,h\in G$ satisfy $q \in V(g)\cap U(h)$. Notice that $V_P(g)\nsubseteq V_P(h)$, for otherwise $h\in \bigcap V_P(g)$, which is contained in $q$ by hypothesis. It follows that there is some $p \in P$ satisfying $g\in p$ and $h\notin p$. Therefore, $p \in V(g)\cap U(h)$.

(4) $\Rightarrow$ (1). Suppose $q$ is in the patch closure of $P$ and let $f\in \widehat{q}+B$. Then $f=\widehat{g}+b$ with $g\in q$. Suppose $e=\widehat{h}+b' \in \widehat{G}+B$ and $e \in f^{\pperp}$. This means that $\sC(e)\subseteq \sC(f)$. Towards a contradiction, assume $h \notin q$. Then $q \in V(g) \cap U(h)$, so by hypothesis there is $p \in P$ such that $p \in V(g) \cap U(h)$. But then there are infinitely many $n \in \N$ such that $f(q,n) = 0 \neq e(q,n)$, which contradicts that $\sC(e) \subseteq \sC(f)$. Hence $h \in q$, $e \in \widehat{q}+B$, and $\widehat{q} +B \in \Spec_d(G)$.
\end{proof}

The following observation will allow us to extend \ref{Pelement}.

\begin{proposition}
Let $G$ be an abelian $\ell$-group and $\emptyset \neq P \subseteq \Spec(G)$ with $\bigcap P = \{ 0 \}$. Then each $q \in \Min(G)$ belongs to the patch closure of $P$.
\end{proposition}

\begin{proof}
Let $q \in \Min(G)$ and take $0<g,h\in G$ for which $q \in U(g)\cap V(h)$, a patch basic open set. Then $h\in q$ means there is some $0<t\in G\ssm q$ such that $h\wedge t=0$ (\cite[Theorem 1.2.11]{AF88}). Replacing $g$ with $t\wedge g$ one may assume, without loss of generality, that $g\wedge h=0$. Since $0<g$ there is some $p \in P$ such that $g\notin p$. Disjointness implies that $h\in p$. Thus, $p\in U(g)\cap V(h)$. Hence, $q$ is in the patch closure of $P$.
\end{proof}

\begin{remark}
Since the $d$-subgroups of $G$ coincide with precisely those primes that belong to the patch closure of $\Min(G)$, it follows that as long as $\bigcap P= \{ 0 \}$, then $\widehat{p}+B \in \Spec_d(\widehat{G}+B)$ whenever $p \in \Spec_d(G)$.
\end{remark}


\section{$G+B$ and $\M$}\label{MfromY}

In this section we will establish conditions which imply $\widehat{G}+B \in \M$, and then use them to produce an example of an $\ell$-group in $\mathbf{W}^* \cap \M$ whose Yosida space is not zero-dimensional.


\begin{theorem}\label{main2}
Let $G$ be an abelian $\ell$-group and $P \subseteq \Spec(G)$ with $\bigcap P = \{ 0 \}$.
The following statements are equivalent.
\begin{enumerate}
\item[(1)] $\widehat{G}+B \in \M$. 
\item[(2)] $P \subseteq \Max(G)$ and $P$ is patch dense in $\Spec(G)$.
\item[(3)] $G \in \Y$ and $P$ is a patch dense subset of $\Max(G)$.
\end{enumerate}
\end{theorem}

\begin{proof}
(1) $\Rightarrow$ (2). Suppose $\widehat{G}+B \in \M$ and let $p \in P$. Towards a contradiction, assume $p \notin \Max(G)$. Then there is some $q \in \Spec(G)$ such that $p<q < G$. Choose $g\in q^+\ssm p$ and let $b = g_{(p,1)}$. Observe that the principal convex $\ell$-subgroup of $\widehat{G}+B$ generated by $b$ is strictly smaller than $b^{\pperp}$: if $h \in G \ssm q$, then $h_{(p,1)} \in b^{\pperp}$ but $h_{(p,1)} \nleq nb$ for any $n \in \N$, which contradicts that $\widehat{G}+B \in \M$. Hence $P \subseteq \Max(G)$.

Next, if $q \in \Spec(G)$, then $\widehat{q} + B \in \Spec(\widehat{G}+B) = \Spec_d(\widehat{G}+B)$ by \ref{spec(G+B)} and \ref{Main} (since $\widehat{G}+B \in \M$), so  $q$ is in the patch closure of $P$ by \ref{Pelement}. Consequently, $P$ is patch dense in $\Spec(G)$.

(2) $\Rightarrow$ (1). Assume that $P\subseteq \Max(G)$ and that $P$ is patch dense in $\Spec(G)$. By \ref{Pelement}, patch density yields that every prime of the form $\widehat{q}+B$ for $q\in \Spec(G)$ is a $d$-subgroup.  Since $P\subseteq \Max(G)$ we know that each $m_{(p,n)}$ is also a $d$-subgroup. So $\Spec(\widehat{G}+B) =\Spec_d(\widehat{G}+B)$ by \ref{spec(G+B)}, and have $G+B \in \M$ by \ref{Main}.

(2) $\Rightarrow$ (3). From (2) we gather that $\Max(G)$ is patch dense in $\Spec(G)$ and hence $G \in \Y$ by \ref{Yosida}. Since $P$ is patch dense in $\Spec(G)$ it is also patch dense in $\Max(G)$.

(3) $\Rightarrow$ (2). Since $G \in \Y$, it follows that $\Max(G)$ is patch dense in $\Spec(G)$. Since $P$ is patch dense in $\Max(G)$ it follows that $P$ is patch dense in $\Spec(G)$.



\end{proof}

\begin{remark}
By \ref{Yosida}(6), if $\widehat{G}+B \in \Y$, then $G \in \Y$. So if $P$ is patch dense in $\Max(G)$, then $G \in \Y$ if and only if $\widehat{G}+B \in \Y$ by  \ref{main2}.
\end{remark}

The next result characterizes those {\bf W}-objects $(G,u)$ for which $(G+B,u)$ is also a {\bf W}-object.

\begin{proposition}\label{G+BinW}
Let $(G,u)\in {\bf W}$.  The following statements are equivalent.
\begin{enumerate}
\item[(1)] $(\widehat{G}+B,u)\in{\bf W}$. 
\item[(2)] There exists $P\subseteq YG \cap \Max(G)$ such that $\bigcap P = \{ 0 \}$. 
\item[(3)] For some set $X$, there exists an embedding of $G$ into $\displaystyle{\R^X}$ such that $u\mapsto {\bf 1}$.
\end{enumerate}
\end{proposition}

\begin{proof}
(1) $\Rightarrow$ (2). Suppose (1) holds. Then $\bigcap P = \{ 0 \}$ by hypothesis and $\widehat{G}+B \in \Arch$. Since $G/p$ embeds into $\widehat{G}+B$ for every $p \in P$, it follows that $G/p \in \Arch$ and $p \in \Max(G)$ for every $p \in P$. Now $u \notin p$ for each $p \in P$ because $u$ is a weak unit in $\widehat{G}+B$, so $P \subseteq YG$. Hence (2) holds.

(2) $\Rightarrow$ (3). Suppose (2) holds. Then $P$ is dense in $YG$. In the Yosida representation of $(G,u)$, for each $g\in G$, one has $P\subseteq g^{-1}(\R)$. Therefore, the Yosida representation followed by the restriction map to $P$ produces the desired embedding of $G$ into $\R^{P}$.

(3) $\Rightarrow$ (1). Given an $\ell$-embedding $\psi \colon G \to \R^X$ with $\psi(u)=1$, it follows that $m_x \in \Max(G)$ for every $x \in X$, where $m_x = \psi^{-1}(\{ f \in \R^X : f(x) = 0 \})$. Set $P=\{m_x: x\in X\}$. Clearly each $m_x$ is a value of $u$ and $\bigcap P = \{ 0 \}$. Constructing $\widehat{G}+B$, we gather that $\widehat{G}+B$ is a archimedean and that $u$ is a weak order unit of $\widehat{G}+B$.
\end{proof}

One of the early questions we had was whether the Yosida space of a $\mathbf{W}^*$-object  in $\M$ is necessarily zero-dimensional. The following example shows that this need not be the case.

\begin{example}
Let $G$ be the $\ell$-group from \cite{mz06} described in Example \ref{MZex} above. Recall that $G \in \Y \ssm \M$. Let $u \in G$ be the function given by $u(t) = t+1$. Then $u$ is a strong unit in $G$ and the  $\mathbf{W}^*$-object $(G,u)$ has $Y(G,u)=[0,\infty]$ (the one-point compactification of the interval $[0,\infty) \subseteq \R$). Take $P=\Max(G)$ and construct $\widehat{G}+B$. Then $\widehat{G}+B \in \M$ by Theorem \ref{main2} and $(\widehat{G}+B,u) \in \bf{W}^*$ by Proposition \ref{G+BinW}. Finally, $Y(G,u)$ is homeomorphic to a subspace of $Y(\widehat{G}+B,u)$ because $\widehat{G} + B$ is a retract of $G$, so $Y(\widehat{G}+B)$ is not zero-dimensional.
\end{example}


\section{Statements and Declarations}

\begin{itemize}
\item The authors contributed equally to the creation of this article.
\item The authors have no competing interests to declare that are relevant to the content of this article.
\end{itemize}

\section{Funding Declarations}
\begin{itemize}
\item Brian Wynne received financial support from PSC-CUNY Research Award \#66070-00 54 and an AMS-Simons Research Enhancement Grant for Primarily Undergraduate Institution Faculty.
\item Papiya Bhattacharjee, Anthony Hager, and Warren McGovern received no financial support.
\end{itemize}

\subsection*{Acknowledgements}

We are grateful to the anonymous reviewer for their detailed and informative report, which helped us to simplify and improve some of our arguments and the overall exposition of the paper.


\end{document}